\documentclass[12pt]{amsart}
\usepackage{amsfonts,amssymb,amscd,amsmath,enumerate,color, xypic}
\def\NZQ{\mathbb}               % the font for N,Z,Q,R,C

\def\ZZ{{\NZQ Z}}
\def\RR{{\NZQ R}}

\def\frk{\mathfrak}               % font for "Fraktur"

\def\Phi{{\frk N}}
\def\opn#1#2{\def#1{\operatorname{#2}}} % to make operators
\opn\chara{char} 
\opn\length{\ell} 
\opn\pd{pd} 
\opn\rk{rk}
\opn\projdim{proj\,dim} 
\opn\injdim{inj\,dim} 
\opn\rank{rank}
\opn\depth{depth} 
\opn\grade{grade} 
\opn\height{height}
\opn\embdim{emb\,dim} 
\opn\codim{codim}

\opn\Tr{Tr} 
\opn\bigrank{big\,rank}
\opn\superheight{superheight}
\opn\lcm{lcm}
\opn\trdeg{tr\,deg}%\emph{
\opn\reg{reg} 
\opn\lreg{lreg} 
\opn\ini{in} 
\opn\lpd{lpd}
\opn\size{size}
\opn\mult{mult}
\opn\dist{dist}
\opn\cone{cone}
\opn\lex{lex}
\opn\rev{rev}
\opn\div{div} \opn\Div{Div} \opn\cl{cl} \opn\Cl{Cl}
\opn\Spec{Spec} \opn\Supp{Supp} \opn\supp{supp} \opn\Sing{Sing}
\opn\Ass{Ass} \opn\Min{Min}
\opn\Ann{Ann} \opn\Rad{Rad} \opn\Soc{Soc}
\opn\Syz{Syz} \opn\Im{Im} \opn\Ker{Ker} \opn\Coker{Coker}
\opn\Am{Am} \opn\Hom{Hom} \opn\Tor{Tor} \opn\Ext{Ext}
\opn\End{End} \opn\Aut{Aut} \opn\id{id} \opn\ini{in}

\opn\nat{nat}
\opn\pff{pf}%   \pf exists already
\opn\Pf{Pf} \opn\GL{GL} \opn\SL{SL} \opn\mod{mod} \opn\ord{ord}
\opn\Gin{Gin}
\opn\Hilb{Hilb}\opn\adeg{adeg}\opn\std{std}\opn\ip{infpt}
\opn\Pol{Pol}
\opn\sat{sat}
\opn\Var{Var}
\opn\Gen{Gen}

\opn\aff{aff} \opn\con{conv} \opn\relint{relint} \opn\st{st}
\opn\lk{lk} \opn\cn{cn} \opn\core{core} \opn\vol{vol}
\opn\link{link} \opn\star{star}
\opn\gr{gr}

\def\Fc{{\mathcal F}}

\def\Pc{{\mathcal P}}
\def\Qc{{\mathcal Q}}

\def\vol{{\textnormal{vol}}}
\def\conv{{\textnormal{conv}}}
\def\pot#1#2{#1[\kern-0.28ex[#2]\kern-0.28ex]}

\opn\dirlim{\underrightarrow{\lim}}
\opn\inivlim{\underleftarrow{\lim}}
\def\Implies{\ifmmode\Longrightarrow \else
	\unskip${}\Longrightarrow{}$\ignorespaces\fi}
\def\implies{\ifmmode\Rightarrow \else
	\unskip${}\Rightarrow{}$\ignorespaces\fi}
\def\iff{\ifmmode\Longleftrightarrow \else
	\unskip${}\Longleftrightarrow{}$\ignorespaces\fi}

\let\:=\colon
\newtheorem{Theorem}{Theorem}[section]
\newtheorem{Lemma}[Theorem]{Lemma}

\newtheorem{Question}[Theorem]{Question}
\newtheorem*{acknowledgement}{Acknowledgment}
\let\epsilon\varepsilon
\let\phi=\varphi
\let\kappa=\varkappa
\def\qed{\ifhmode\textqed\fi
	\ifmmode\ifinner\quad\qedsymbol\else\dispqed\fi\fi}
\def\textqed{\unskip\nobreak\penalty50
	\hskip2em\hbox{}\nobreak\hfil\qedsymbol
	\parfillskip=0pt \finalhyphendemerits=0}
\def\dispqed{\rlap{\qquad\qedsymbol}}

\opn\dis{dis}
\opn\height{height}
\opn\dist{dist}
\def\pnt{{\raise0.5mm\hbox{\large\bf.}}}

\opn\Lex{Lex}
\opn\conv{conv}

\begin{document}
	
	\title{Flat $\delta$-vectors and their Ehrhart polynomials}
	\author[T. Hibi]{Takayuki Hibi}
	\address[Takayuki Hibi]{Department of Pure and Applied Mathematics,
		Graduate School of Information Science and Technology,
		Osaka University,
		Suita, Osaka 565-0871, Japan}
	\email{hibi@math.sci.osaka-u.ac.jp}
	\author[A. Tsuchiya]{Akiyoshi Tsuchiya}
	\address[Akiyoshi Tsuchiya]{Department of Pure and Applied Mathematics,
		Graduate School of Information Science and Technology,
		Osaka University,
		Suita, Osaka 565-0871, Japan}
	\email{a-tsuchiya@cr.math.sci.osaka-u.ac.jp}
	
	\subjclass[2010]{52B05, 52B20}
	\date{}
	\keywords{Ehrhart polynomial, $\delta$-vector, integral convex polytope}
	
	\begin{abstract}
	We call the $\delta$-vector of an integral convex polytope of dimension $d$ flat if the $\delta$-vector is of the form $(1,0,\ldots,0,a,\ldots,a,0,\ldots,0)$, where $a \geq 1$. In this paper, we give the complete characterization of possible flat $\delta$-vectors. Moreover, for an integral convex polytope $\mathcal{P} \subset \mathbb{R}^N$ of dimension $d$, we let $i(\mathcal{P},n)=|n\mathcal{P} \cap \mathbb{Z}^N|$ and $\ i^*(\mathcal{P},n)=|n(\mathcal{P} \setminus \partial \mathcal{P}) \cap \mathbb{Z}^N|.$ By this characterization, we show that for any $d \geq 1$ and for any $k,\ell \geq 0$ with $k+\ell \leq d-1$, there exist integral convex polytopes $\mathcal{P}$ and $\mathcal{Q}$ of dimension $d$ such that (i) For $t=1,\ldots,k$, we have $i(\mathcal{P},t)=i(\mathcal{Q},t),$ (ii) For $t=1,\ldots,\ell$, we have $i^*(\mathcal{P},t)=i^*(\mathcal{Q},t)$ and (iii) $i(\mathcal{P},k+1) \neq i(\mathcal{Q},k+1)$ and $i^*(\mathcal{P},\ell+1)\neq i^*(\mathcal{Q},\ell+1).$
	\end{abstract}
	 \maketitle
	 \section*{Introduction}
	 Let $\Pc \subset \RR^N$ be an integral convex polytope of dimension $d$ and $\partial \Pc$ its boundary.
	 Here an integral convex polytope is a convex polytope all of whose vertices have integer coordinates.
	 For $n=1,2,\ldots$, we write
	 \[i(\Pc,n)=|n\Pc \cap \ZZ^N|, \  \ i^*(\Pc,n)=|n(\Pc \setminus \partial \Pc) \cap \ZZ^N|, \]
	 where $n\Pc = \{ \, n \alpha \, : \, \alpha \in \Pc \, \}$ and $|X|$ is the cardinality of a finite set $X$.
	 The enumerative function $i(\Pc,n)$ has the following fundamental properties,
	 which were studied originally in the work of Ehrhart \cite{Ehrhart}:
	 \begin{itemize}
	 	\item $i(\Pc,n)$ is a polynomial in $n$ of degree $d$;
	 	\item $i(\Pc,0)=1$;
	 	\item (loi de r\'{e}ciprocit\'{e}) $i^*(\Pc,n)=(-1)^di(\Pc,-n)$ for every integer $n>0$.
	 \end{itemize}
	 This polynomial $i(\Pc,n)$ is called the \textit{Ehrhart polynomial} of $\Pc$.
	 Consult 
	 \cite[Part II]{HibiRedBook} and \cite[pp.~235--241]{StanleyEC}
	 for fundamental materials on Ehrhart polynomials.
	 
	  We define the sequence $\delta_0,\delta_1, \delta_2, \ldots$ of integers by the formula
	  \[ (1-\lambda)^{d+1} \sum\limits_{n=0}^{\infty} i(\Pc,n)\lambda^n=\sum\limits_{j=0}^{\infty} \delta_j \lambda^j.\]
	  Since $i(\Pc,n)$ is a polynomial in $n$ of degree $d$, a fundamental fact on generating
	  functions (\cite[Corollary 4.3.1]{StanleyEC}) guarantees that $\delta_j=0$ for every $j>d$.
	  The sequence $\delta(\Pc)=(\delta_0,\delta_1,\ldots,\delta_d)$ is called the $\delta$-vector of $\Pc$.
	  The following properties on $\delta$-vectors are well known:
	  \begin{itemize}
	  	\item $\delta_0=1, \delta_1=|\Pc \cap \ZZ^N|-(d+1)$ and $\delta_d=|(\Pc\setminus\partial\Pc)\cap \ZZ^N|$;
	  	\item $i(\Pc,n)=\sum\limits_{j=0}^{d}\binom{n+d-j}{d}\delta_j;$
	  	\item Each $\delta_i$ is nonnegative (\cite{StanleyDRCP});
	  	\item Let $s=\text{max}\{i : \delta_i \neq 0\}$. Then for $t=1,\ldots, d-s$, we have $i^*(\Pc,t)=0$ and $i^*(\Pc,d-s+1)=\delta_s$;
	  	\item If $N=d$, the leading coefficient $(\sum_{i=0}^d\delta_i)/d!$ of $i(\Pc,n)$ 
	  	is equal to the usual volume of $\Pc$ (\cite[Proposition 4.6.30]{StanleyEC}). 
	  	In general, the positive integer $\vol(\Pc) = \sum_{i=0}^d\delta_i$ 
	  	is said to be the {\em normalized volume} of $\Pc$. 
	  \end{itemize}
	 Through this paper, we assume that $N=d$.
	 
	 We call the $\delta$-vector of an integral convex polytope of dimension $d$ \textit{flat} if the $\delta$-vector is of the form $(1,0,\ldots,0,a,\ldots,a,0,\ldots,0)$, where $a \geq 1$.  
	 In this paper, we will give the complete characterization of possible flat $\delta$-vectors.
	 In fact, we show the following theorem.
	 \begin{Theorem}
	 	\label{joint}
	 	Let $d \geq 1$ and $k, \ell \geq 0$ with $k+\ell \leq d-1$ and $a \geq 1$.
	 	Given a finite sequence %$(\delta_0,\ldots,\delta_d)$ of nonnegative integers which satisfies
	 	%\begin{displaymath}
	 	%\delta_i=\left\{
	 	%\begin{aligned}
	 	%&1,\ \textnormal{if} \ i=0,\\
	 	%&a,\ \textnormal{if} \ k+1\leq i \leq d-\ell, \\ 
	 	%&0,\ \textnormal{otherwise}, \\ 
	 	%\end{aligned}
	 	%\right.
	 	%\end{displaymath} 
	 	$$(\delta_0,\ldots,\delta_d)
	 	=(1, \underbrace{0,\ldots,0}_k,a,\ldots,a,\underbrace{0,\ldots,0}_{\ell}),$$ 
	 	there exists an integral convex polytope $\Pc \subset \RR^d$ of dimension $d$ whose $\delta$-vector coincides with $(\delta_0,\ldots,\delta_d)$ if and only if
	 	$k \leq \ell$.
	 \end{Theorem}
	 This Theorem is a generalization of \cite[Theorem 2.1]{Higashitani}.
	 
	 Moreover, we consider the Ehrhart polynomials of flat $\delta$-vectors.
	 Let $k$ and $\ell$ be positive integers, and
	 let $\Pc$ and $\Qc$ be integral convex polytopes of dimension $d$
	 such that the following conditions are satisfied:
	 \begin{itemize}
	 	\item For $t=1,\ldots,k$, we have $i(\Pc,t)=i(\Qc,t)$;
	 	\item For $t=1,\ldots,\ell$, we have $i^*(\Pc,t)=i^*(\Qc,t)$. 
	 \end{itemize}
	 Since the degree of Ehrhart polynomials equal the dimension of underlying integral convex polytopes and the constant equals 1, and by Ehrhart reciprocity, if $k+\ell \geq d$,
	 then we know that $\Pc$ and $\Qc$ have a common Ehrhart polynomial.
	 However, if $k+\ell \leq d-1$, then $\Pc$ and $\Qc$ don't necessarily have a common Ehrhart polynomial. 
	 By the characterization of flat $\delta$-vectors, we will show the following theorems.
	 \begin{Theorem}
	 	\label{main}
	 	Let $d \geq 1$.
	 	Then for any $k,\ell \geq 0$ with $k+\ell \leq d-1$,  there exist integral convex polytopes $\Pc$ and $\Qc$ of dimension $d$ such that
	 	the followings are satisfied:
	 	\begin{itemize}
	 		\item For $t=1,\ldots,k$, we have $i(\Pc,t)=i(\Qc,t);$\\
	 		\item For $t=1,\ldots,\ell$, we have $i^*(\Pc,t)=i^*(\Qc,t);$\\
	 		\item $i(\Pc,k+1) \neq i(\Qc,k+1)$ and $i^*(\Pc,\ell+1)\neq i^*(\Qc,\ell+1).$
	 	\end{itemize}
	 \end{Theorem}
	 \begin{Theorem}
	\label{main2}
	Let $d \geq 1$.
	Then for any $0 \leq k \leq \ell \leq d-k-1$,  there exists an infinite family $\{\Pc_1,\Pc_2,\ldots\}$ of integral convex polytopes of dimension $d$ such that
	for each $\Pc_i$ and $\Pc_j$ with $i \neq j$, the followings are satisfied:
	\begin{itemize}
		\item For $t=1,\ldots,k$, we have $i(\Pc_i,t)=i(\Pc_j,t);$\\
		\item For $t=1,\ldots,\ell$, we have $i^*(\Pc_i,t)=i^*(\Pc_j,t);$\\
		\item $i(\Pc_i,k+1) \neq i(\Pc_j,k+1)$ and $i^*(\Pc_i,\ell+1)\neq i^*(\Pc_j,\ell+1).$
	\end{itemize}
\end{Theorem} 
	 In Section 1, we recall  the some properties and the calculation method on the $\delta$-vectors of integral simplices.
	 In Section 2, we prove Theorem \ref{joint}, \ref{main} and \ref{main2}.
	 
\section{Preliminaries}
	At first, we recall some properties of $\delta$-vectors.
	There are two well-known inequalities on $\delta$-vectors.
	Let $s = \max\left\{i : \delta_i \neq 0 \right\}$.
	One inequality is
	\begin{align}
		\delta_0+\delta_1+\cdots+\delta_{i}\leq \delta_s+\delta_{s-1}+\cdots+\delta_{s-i}, \ \ 0 \leq i \leq \lfloor s/2 \rfloor,
	\end{align}
	which was proved by Stanley \cite{RS_OHF}, and another one is
	\begin{align}
		\delta_d+\delta_{d-1}+\cdots+\delta_{d-i}\leq \delta_1+\delta_2+\cdots+\delta_{i+1}, \ \ 0 \leq i \leq\lfloor(d-1)/2\rfloor ,
	\end{align}
	which appears in the work of the first author \cite[Remark 1.4]{HibiLBT}.
	Also, there are more recent and more general results on $\delta$-vectors by Alan Stapledon in \cite{Stap}.
	
	Moreover, we recall the following lemma.
	\begin{Lemma}
		\label{delta}
	Suppose that $(\delta_0, \delta_1, \ldots,\delta_d)$ is the $\delta$-vector of an integral convex polytope of dimension $d$.
	Then there exists an integral convex polytope of dimension $d+1$ whose $\delta$-vector is $(\delta_0, \delta_1, \ldots,\delta_d,0)$.
\end{Lemma}

Next, we recall %from \cite[Part II]{HibiRedBook} 
the well-known combinatorial technique how to compute 
the $\delta$-vector of an integral simplex. 
Given an integral simplex $\Fc \subset \RR ^d$ of dimension $d$ 
with the vertices $v_0, v_1, \ldots, v_d \in \RR^d$, we set 

$$\textnormal{Box}(\Fc)=\left\{ \alpha \in \ZZ^{d+1} : \alpha = \sum_{i=0}^d \lambda_i(v_i,1), \;\;
	0 \leq \lambda_i < 1 \right\}. $$

We define the degree of $\alpha=\sum_{i=0}^{d}\lambda_i(v_i,1) \in \text{Box}(\Fc) \cap \ZZ^d$ 
with $\deg(\alpha)=\sum_{i=0}^d \lambda_i$, i.e., the last coordinate of $\alpha$. 
Then we have the following lemma.
\begin{Lemma}\label{compute}
	Let $\delta(\Fc)=(\delta_0,\delta_1,\ldots,\delta_d).$ 
	Then each $\delta_i$ is equal to the number of integer points $\alpha \in \textnormal{Box}(\Fc)$ with $\deg(\alpha)=i$. 
\end{Lemma}
In particular, if $v_0=(0,0,\ldots,0)$, then for $\alpha=\sum_{i=0}^{d}\lambda_i(v_i,1) \in \textnormal{Box}(\Fc) \cap \ZZ^d$, we have $\deg(\alpha)= \left\lceil \sum_{i=1}^d \lambda_i \right\rceil$.

\section{proofs of Theorems}
At first, in order to prove Theorem \ref{joint}, we show the following lemmas.
\begin{Lemma}
\label{even}
Let $d \geq 3$.
For any $1 \leq k \leq \lfloor (d-1)/2 \rfloor$ and for any $a \geq 1$,
there exists an integral convex polytope $\Pc$ of dimension $d$ such that 

	$$\delta(\Pc)
	=(1,\underbrace{0,\ldots,0}_{k}, a,\ldots,a,\underbrace{0,\ldots,0}_{k}).$$ 
\end{Lemma}

\begin{proof}
We set $\Pc=\text{conv}(\{v_0,\ldots,v_d\}) \subset \RR^d$, where 
\begin{displaymath}
v_i=\left\{
\begin{aligned}
&(0,\ldots ,0),& \ \textnormal{if}& \ i=0,\\
&e_i,&\ \textnormal{if} &\  1 \leq i \leq d-1, \\
&\sum\limits_{j=1}^{d-k} e_j +a(d-2k) \sum\limits_{j=d-k+1}^{d-1} e_j+(a(d-2k) +1)e_d, &\ \textnormal{if} &\  i=d,
\end{aligned}
\right.
\end{displaymath} 
where $e_{1}, \ldots, e_{d}$ are the canonical unit coordinate vectors of $\RR^{d}$.  
We compute the $\delta$-vector of $\Pc$.
Let $\lambda_1, \ldots, \lambda_d \in [0,1)$ such that $\sum_{i=1}^d \lambda_i v_i \in \ZZ^d$. 
Then there exists an integer $t$ with $1\leq t \leq a(d-2k)$ such that $\lambda_d=\cfrac{t}{a(d-2k) +1}$. 
Hence we have  
\begin{displaymath}
\lambda_i=\left\{
\begin{aligned}
&\cfrac{a(d-2k) +1-t}{a(d-2k) +1},&\ \textnormal{if}& \ 1 \leq i \leq d-k ,\\
&\cfrac{t}{a(d-2k) +1},&\  \textnormal{if}& \ d-k +1 \leq i \leq d-1.\\
\end{aligned}
\right.
\end{displaymath}
For  $1\leq t \leq a(d-2k)$, we let $f(t)=\cfrac{a(d-2k) +1-t}{a(d-2k) +1}(d-k)+\cfrac{kt}{a(d-2k) +1}$.
Then we have
$f(t)=d-k -\cfrac{t(d-2k)}{a(d-2k) +1}.$
Let $0 \leq \ell \leq d-2k-1 $ and %$a\ell+1 \leq t \leq a(\ell+1)$.
$1 \leq j \leq a$.
Since 
\begin{displaymath}
\begin{aligned}
f(a\ell +j)&=d-k -\cfrac{(a\ell +j)(d-2k)}{a(d-2k) +1}\\
&=d-k -\cfrac{\ell(a(d-2k)+1)-\ell +j(d-2k)}{a(d-2k) +1}\\
&=d-k -\ell -\cfrac{j(d-2k)-\ell}{a(d-2k) +1},
\end{aligned}
\end{displaymath}

%{displaymath}
%\begin{aligned}
%f(a\ell +1)&=d-k -\cfrac{(a\ell +1)(d-2k)}{a(d-2k) +1}\\
%&=d-k -\cfrac{\ell(a(d-2k)+1)-\ell +(d-2k)}{a(d-2k) +1}\\
%&=d-k -\ell -\cfrac{d-2k-\ell}{a(d-2k) +1}\\
%\end{aligned}
%\end{displaymath}
%and
%\begin{displaymath}
%\begin{aligned}
%f(a(\ell +1))&=d-k -\cfrac{a(\ell +1)(d-2k)}{a(d-2k) +1}\\
%&=d-k -\cfrac{\ell(a(d-2k)+1)-\ell +a(d-2k)}{a(d-2k) +1}\\
%&=d-k -\ell -\cfrac{a(d-2k)-\ell}{a(d-2k) +1}\\
%\end{aligned}
%\end{displaymath}
we have $\lceil f(a\ell+j) \rceil =d-k-\ell$.
Hence we know that
\begin{displaymath}
\delta_i=\left\{
\begin{aligned}
&1, \ \textnormal{if} \ i=0,\\
&a, \ \textnormal{if} \ k +1 \leq i \leq d-k, \\
&0, \ \textnormal{otherwise},
\end{aligned}
\right.
\end{displaymath} 
as desired.
\end{proof}
\begin{Lemma}
	\label{full}
	For any $d \geq 1$ and for any $a \geq 1$, there exists an integral convex polytope $\Pc \subset \RR^d$ such that
	$\delta(\Pc)
	=(1,a,\ldots,a).$
\end{Lemma}
\begin{proof}
	We set $\Pc=\text{conv}(\{v_0,\ldots,v_d\}) \subset \RR^d$, where 
	\begin{displaymath}
	v_i=\left\{
	\begin{aligned}
	&(0,\ldots, 0),& \ \textnormal{if} &\ i=0,\\
	&e_i,& \ \textnormal{if} &\ 1 \leq i \leq d-1, \\
	&ad\sum\limits_{j=1}^{d-1} e_j +(ad +1)e_d,& \ \textnormal{if} &\  i=d.
	\end{aligned}
	\right.
	\end{displaymath} 
	We compute the $\delta$-vector of $\Pc$.
	Let $\lambda_1, \ldots, \lambda_d \in [0,1)$ such that $\sum_{i=1}^d \lambda_i v_i \in \ZZ^d$. 
	Then there exists an integer $t$ with $1\leq t \leq ad$ such that $\lambda_d=\cfrac{t}{ad +1}$. 
	Hence for $1 \leq i \leq d-1$, we have
	$\lambda_i=\cfrac{t}{ad+1}.$
	For  $1\leq t \leq ad$, we let $f(t)=\cfrac{dt}{ad +1}$,
	$0 \leq k \leq d-1  $ and %$ka+1 \leq t \leq (k+1)a$.
	$1 \leq j \leq a$.
	Since 
	$$f(ka+j)=\cfrac{d(ka+j)}{ad +1}
		=\cfrac{k(ad+1)-k+dj}{ad+1}
		=k+\cfrac{jd-k}{ad +1},$$
%	$$f(ka+1)=\cfrac{d(ka+1)}{ad +1}
%	=\cfrac{k(ad+1)-k+d}{ad+1}
%	=k+\cfrac{d-k}{ad +1}$$
%	and
%	$$f((k+1)a)=\cfrac{ad(k+1)}{ad +1}=\cfrac{k(ad+1)-k+ad}{ad +1}=k+\cfrac{ad-k}{ad +1},$$
	we have $\lceil f(ka+j) \rceil=k+1$.
	Hence we know that
	\begin{displaymath}
	\delta_i=\left\{
	\begin{aligned}
	&1,& \ \textnormal{if} &\ i=0,\\
	&a,& \ \textnormal{if} &\ 1 \leq i \leq d. \\
	\end{aligned}
	\right.
	\end{displaymath} 
	as desired.
	\end{proof}
By using Lemma  \ref{delta}, \ref{even} and \ref{full}, and Hibi's inequality, 
the assertion of Theorem \ref{joint} follows.

Next, we prove Theorem \ref{main}.
\begin{proof}[Proof of Theorem \ref{main}]
By Theorem \ref{joint}, there exist integral convex polytopes $\Pc$ and $\Qc$ of dimension $d$ such that 
\begin{displaymath}
\delta_i(\Pc)=\left\{
\begin{aligned}
&1,\ \textnormal{if} \ i=0,\\
&a,\ \textnormal{if} \ 1\leq i \leq k,\\
&0,\ \textnormal{if} \ \textnormal{otherwise},
\end{aligned}
\right.
\end{displaymath} 
and
\begin{displaymath}
\delta_i(\Qc)=\left\{
\begin{aligned}
&1,\ \textnormal{if} \ i=0,\\
&a,\ \textnormal{if} \ 1 \leq i \leq d-\ell, \\
&0,\ \textnormal{if} \ \textnormal{otherwise},
\end{aligned}
\right.
\end{displaymath}
where $a \geq 1.$ 
Then we know that for $t=1,\ldots,\ell$, $i^*(\Pc,t)=i^*(\Qc,t)=0$ 
and $0=i^*(\Pc,\ell+1)\neq i^*(\Qc,\ell +1)=a$.
Since
$$i(\Pc,n)=\binom{n+d}{d}+a\sum\limits_{i=1}^{k}\binom{n+d-i}{d}$$
and
$$i(\Qc,n)=\binom{n+d}{d}+a\sum\limits_{i=1}^{d-\ell}\binom{n+d-i}{d}$$
and  since for $k+1 \leq i \leq d-\ell$ and for $1 \leq t \leq k$, we have $\binom{t+d-i}{d}=0$,
we know that for $t=1,\ldots, k$, we have $i(\Pc,t)=i(\Qc,t)$.
Moreover, since $\sum\limits_{i=k+1}^{d-\ell}\binom{k+1+d-i}{d}=1$,
we have $i(\Pc,k+1) \neq i(\Qc,k+1)$, as desired.
\end{proof}

Next, we prove Theorem \ref{main2}.
\begin{proof}[Proof of Theorem \ref{main2}]
	By Proposition \ref{joint}, there exist integral convex polytopes $\Pc$ and $\Qc$ of dimension $d$ such that 
	\begin{displaymath}
	\delta_i(\Pc)=\left\{
	\begin{aligned}
	&1,\ \textnormal{if} \ i=0,\\
	&a,\ \textnormal{if} \ k+1 \leq i \leq d-\ell, \\
	&0,\ \textnormal{if} \ \textnormal{otherwise},
	\end{aligned}
	\right.
	\end{displaymath} 
	and
	\begin{displaymath}
	\delta_i(\Qc)=\left\{
	\begin{aligned}
	&1,\ \textnormal{if} \ i=0,\\
	&b,\ \textnormal{if} \ k+1 \leq i \leq d-\ell, \\
	&0,\ \textnormal{if} \ \textnormal{otherwise},
	\end{aligned}
	\right.
	\end{displaymath} 
	where $1 \leq a < b$.
	Then we know that for $t=1,\ldots,\ell$, $i^*(\Pc,t)=i^*(\Qc,t)=0$ 
	and $a=i^*(\Pc,\ell+1)\neq i^*(\Qc,\ell +1)=b$.
	Since
	$$i(\Pc,n)=\binom{n+d}{d}+a\sum\limits_{i=k+1}^{d-\ell}\binom{n+d-i}{d}$$
	and
	$$i(\Qc,n)=\binom{n+d}{d}+b\sum\limits_{i=k+1}^{d-\ell}\binom{n+d-i}{d}$$
	and  since for $k+1 \leq i \leq d-\ell$ and for $1 \leq t \leq k$, we have $\binom{t+d-i}{d}=0$,
	we know that for $t=1,\ldots, k$, we have $i(\Pc,t)=i(\Qc,t)$.
	Moreover, since $\sum\limits_{i=k+1}^{d-\ell}\binom{k+1+d-i}{d}=1$,
	we have $i(\Pc,k+1) \neq i(\Qc,k+1)$.
	\end{proof}
	
	By using integral convex polytopes with flat $\delta$-vectors, we can construct an infinite family in 
	Theorem \ref{main2}.
	However, for $0 \leq \ell < k \leq d-\ell-1$, we cannot construct such an infinite fimily.
	Finally, we give the following question.
	\begin{Question}
		 	Let $d \geq 1$.
		 	Then for any $0 \leq \ell < k \leq d-\ell-1$,  does there exist an infinite family $\{\Pc_1,\Pc_2,\ldots\}$ of integral convex polytopes of dimension $d$ such that
		 	for each $\Pc_i$ and $\Pc_j$ with $i \neq j$, the followings are satisfied:
		 	\begin{itemize}
		 		\item For $t=1,\ldots,k$, we have $i(\Pc_i,t)=i(\Pc_j,t);$\\
		 		\item For $t=1,\ldots,\ell$, we have $i^*(\Pc_i,t)=i^*(\Pc_j,t);$\\
		 		\item $i(\Pc_i,k+1) \neq i(\Pc_j,k+1)$ and $i^*(\Pc_i,\ell+1)\neq i^*(\Pc_j,\ell+1)$?
		 	\end{itemize}
		 \end{Question}
	\begin{acknowledgement}
		The authors would like to thank anonymous referees for reading the manuscript carefully.
	\end{acknowledgement}
	
\end{document}